%% file: cpt-sheaves.tex
\documentclass{amsart}

\usepackage{amsmath,amsthm,hyperref,amssymb,mathtools,enumitem,tikz-cd}
\usepackage{mathrsfs}
\usepackage{graphicx} 

\newcommand{\Z}{\mathbb Z}
\newcommand{\supp}{\operatorname{supp}}
\newcommand{\Shv}{\operatorname{Shv}}

\newcommand{\stabPr}{\mathcal P\mathrm r_{\mathrm{stab}}}

\newcommand{\inftyCats}{\widehat{\mathcal C\mathrm{at}}_\infty}

\DeclareMathOperator{\colimit}{\operatorname{colim}}
\DeclareMathOperator{\limit}{\operatorname{lim}}

\title{Compact sheaves on a locally compact space}
\author{Oscar Harr}
\date{\today}

\newtheorem{thm}{Theorem}[section]
\newtheorem{lem}[thm]{Lemma}
\newtheorem{prop}[thm]{Proposition}
\newtheorem{cor}[thm]{Corollary}

\theoremstyle{definition}
\newtheorem{defn}[thm]{Definition}
\newtheorem{rmk}[thm]{Remark}

\begin{document}

\maketitle

\begin{abstract}
  We describe the compact objects in the $\infty$-category of $\mathcal C$-valued sheaves $\Shv (X,\mathcal C)$ on a hypercomplete locally compact Hausdorff space $X$,
  for $\mathcal C$ a compactly generated stable $\infty$-category.
  When $X$ is a non-compact connected manifold and $\mathcal C$ is the unbounded derived category of a ring, our result recovers a result of Neeman.
  Furthermore, for $X$ as above and $\mathcal C$ a nontrivial compactly generated stable $\infty$-category, we show that $\Shv (X,\mathcal C)$ is compactly generated if and only if $X$ is totally disconnected.
\end{abstract}
The aim of this note is to clarify and expand on a point made by Neeman \cite{neeman2001}. Let $M$ be a non-compact connected manifold, and let $\Shv (M,\mathcal D(\Z ))$ denote the unbounded derived category of sheaves of abelian groups on $M$. Neeman shows that, up to equivalence, the only compact object in $\Shv (M,\mathcal D(\Z ))$ is the zero sheaf. This implies that $\Shv (M,\mathcal D(\Z ))$ is very far from compactly generated. Nevertheless, it follows from Lurie's covariant Verdier duality theorem \cite[Thm~5.5.5.1]{HA} that $\Shv (M,\mathcal D(\Z ))$ satisfies a related smallness condition: it is \emph{dualizable} in the symmetric monoidal $\infty$-category $\stabPr^{\otimes}$ of stable presentable $\infty$-categories,
which holds more generally if $M$ is replaced with any locally compact Hausdorff space $X$.
Although every compactly generated presentable stable $\infty$-category is dualizable \cite[Prop~D.7.2.3]{SAG}, Neeman's example thus shows that the converse is false.
The existence of this large and interesting class of stable presentable $\infty$-categories that are dualizable but not compactly generated forms part of the motivation behind Efimov's continuous extensions of localizing invariants \cite{efimov-ICM}, see also \cite{hoyois}.

This note is concerned with the following two questions about the $\infty$-category of $\mathcal C$-valued sheaves on a general locally compact Hausdorff space $X$, where $\mathcal C$ is some compactly generated stable $\infty$-category (e.g. the unbounded derived $\infty$-category of a ring or the $\infty$-category of spectra):
\begin{enumerate}[label=(\arabic*)]
\item How rare is it for $\Shv (X,\mathcal C)$ to be compactly generated?
\item How far is $\Shv (X,\mathcal C)$ from being compactly generated in general?
\end{enumerate}
With a relatively mild completeness assumption on $X$ (see Section~\ref{sec:hypercompleteness}), we answer question (2) by showing that a $\mathcal C$-valued sheaf $\mathscr F$ on $X$ is compact as an object of $\Shv (X,\mathcal C)$ if and only if it has compact support, compact stalks, and is locally constant (Theorem~\ref{thm:cpt-sheaves}). Thus if $X$ is for instance a CW complex, the subcategory of compact objects $\Shv (X,\mathcal C)^\omega$ remembers only the \emph{homotopy type} of the compact path components of $X$, and it is therefore impossible to reconstruct the entire sheaf category $\Shv (X,\mathcal C)$, or equivalently the homeomorphism type of $X$, from this information.

In his 2022 ICM talk, Efimov mentions that the $\infty$-category of $\mathcal D(R)$-valued sheaves on a locally compact Hausdorff space $X$ `is almost never compactly generated (unless $X$ is totally disconnected, like a Cantor set)' \cite[slide~13]{efimov-ICM}. As a corollary to our description of the compact objects of $\Shv (X,\mathcal C)$, we verify--modulo the same completeness assumption mentioned above--that indeed the \emph{only} locally compact Hausdorff spaces $X$ with $\Shv (X,\mathcal C)$ compactly generated, for some nontrivial $\mathcal C$, are the totally disconnected ones (Proposition~\ref{prop:when-cpt}), thereby answering question (1).

\subsection*{Notation and conventions}
Throughout this note, we use the theory of higher categories and higher algebra, an extensive textbook account of which can be found in the work of Lurie \cite{HTT,HA,SAG}. 
We will also make frequent use of the six-functor formalism for derived sheaves on locally compact Hausdorff spaces, as described classically by \cite{verdier1965dualite} and with general coefficients by \cite{volpe23}.

For convenience, we assume the existence of an uncountable Grothendieck universe $\mathcal U$ of \emph{small} sets and further Grothendieck universes
$\mathcal U'$ and $\mathcal U''$
of \emph{large} and \emph{very large} sets respectively, such that $\mathcal U\in\mathcal U'\in\mathcal U''$.
`Topological space' always implicitly refers to a small topological space, and similarly with `spectrum'. On the other hand, an `$\infty$-category' is an `$(\infty ,1)$-category' is a quasicategory, which unless otherwise stated is large. We let $\inftyCats$ denote the very large $\infty$-category of (large) $\infty$-categories.

Because we are dealing with sheaves on topological spaces, we deem it best to make a clear distinction between actual topological spaces on the one hand, and on the other hand the objects of the $\infty$-category $\mathcal S$ of `spaces' in the sense of Lurie. Following the convention suggested in \cite{cesnavicius-scholze}, we will refer to the latter as \emph{anima}.

Given an $\infty$-category $\mathcal C$, we let $\mathcal C^{\omega}\subseteq\mathcal C$ denote the subcategory spanned by the compact objects. Recall that an object $C\in\mathcal C$ is said to be \emph{compact} if the presheaf of large anima $D\mapsto\operatorname{Map}_{\mathcal C}(C,D)$ preserves small filtered colimits.

\subsection*{Acknowledgements} 
I was partially supported by the Danish National Research Foundation
through the Copenhagen Centre for Geometry and Topology (DRNF151).
I am grateful to Marius Verner Bach Nielsen for comments on the draft,
and to Jesper Grodal and Maxime Ramzi for valuable discussions about the arguments appearing in this note,
and to the latter for pointing out that the original version of Lemma~\ref{lem:small-compare} was false
in the generality in which I had stated it.

\section{$\mathcal C$-hypercomplete spaces}
\label{sec:hypercompleteness}
Given an $\infty$-category $\mathcal C$ and a topological space $X$, we let $\Shv (X,\mathcal C)$ denote the $\infty$-category of $\mathcal C$-valued sheaves on $X$ in the sense of Lurie \cite{HTT}.
That is, $\Shv (X,\mathcal C)$ is the full subcategory of the presheaf $\infty$-category $\operatorname{Fun}(\operatorname{Open}(X)^{\text{op}},\mathcal C)$ consisting of
presheaves $\mathscr F$ satisfying the \emph{sheaf condition}: for any open set $U\subseteq X$ and any open cover $\lbrace U_i\to U\rbrace_{i\in I}$, the canonical map
$$
\mathscr F(U)\to\limit_V\mathscr F(V)
$$
is an equivalence, where $V$ ranges over open sets $V\subseteq U_i\subseteq X$, $i\in I$, considered as a poset under inclusion.
When $\mathcal C = \mathcal S$ is the $\infty$-category of anima, we will abbreviate $\Shv (X) = \Shv (X,\mathcal S)$.

\begin{rmk}
  When $\mathcal C = \mathcal D(R)$ is the unbounded derived $\infty$-category of a ring, the $\infty$-category $\Shv (X,\mathcal D(R))$ is related to, but generally not the same as, the derived $\infty$-category $\mathcal D(\Shv (X,R))$ of the ordinary category of sheaves of $R$-modules on $X$, which is the object studied (via its homotopy category) by Neeman \cite{neeman2001}. However, they do coincide under the completeness assumption that we will impose on $X$ below, see \cite[Prop~7.1]{scholze-6FF}. Since this completeness assumption is verified when $X$ is a topological manifold, our results include those of Neeman.
\end{rmk}

We are interested in topological spaces satisfying the following condition:
\begin{defn}\label{defn:hypercomplete}
  A topological space $X$ is \emph{$\mathcal C$-hypercomplete} if the stalk functors $x^*\colon\Shv (X,\mathcal C)\to\mathcal C$ are jointly conservative for $x$ ranging over the points of $X$.
\end{defn}
The reason for our choice of terminology is that $X$ is $\mathcal S$-hypercomplete if and only if the $0$-localic $\infty$-topos $\Shv (X)$ has enough points, which is equivalent to $\Shv (X)$ being hypercomplete as an $\infty$-topos by Claim (6) in \cite[\S~6.5.4]{HTT}. (This is \emph{not} true for arbitrary $\infty$-topoi, i.e. there are hypercomplete $\infty$-topoi that do not have enough points.)
This subtlety, whereby a morphism of sheaves may fail to be an equivalence even though it is so on all stalks, does not occur for non-derived sheaves: the $1$-topos $\Shv (X,\mathcal S_{\leq 0})$ of sheaves of sets on a topological space $X$ \emph{always} has enough points.
We refer to \cite[\S~6.5.4]{HTT} for a discussion of why it is often preferable to work with non-hypercomplete sheaves, rather than, say, imposing hypercompleteness by replacing $\Shv (X)$ with its hypercompletion $\Shv (X)^\wedge$.

\medskip

Several classes of topological spaces are known to be $\mathcal S$-hypercomplete, and hence also $\mathcal C$-hypercomplete for any compactly generated $\infty$-category $\mathcal C$.\footnote{Indeed, for any such $\mathcal C$ there is a conservative functor
  $$
  \Shv (X,\mathcal C)\to
  \prod_{C\in\mathcal C^\omega}\Shv (X)
  $$
  given informally by mapping $\mathscr F$ to $(\mathscr F_C)_{C\in\mathcal C^\omega}$, where
  $\mathscr F_C = \operatorname{Map}_{\mathcal C}(C,{-})\circ\mathscr F$, which is a sheaf since $\operatorname{Map}_{\mathcal C}(C,{-})$ preserves limits. Also, since $C$ is compact, we have a canonical equivalence $(\mathscr F_C)_x\simeq\operatorname{Map}_{\mathcal C}(C,\mathscr F_x)$ natural in $\mathscr F$ for each $x\in X$, and it follows that if $X$ is $\mathcal S$-hypercomplete then it is also $\mathcal C$-hypercomplete.}
Although only the first two are relevant for this note, here is a list of some classes of topological spaces that have this property:
\begin{itemize}
\item paracompact spaces that are locally of covering dimension $\leq n$ for some fixed $n$ \cite[Cor~7.2.1.12]{HTT},
\item arbitrary CW complexes \cite{hoyois-cw}, 
\item  finite-dimensional Heyting spaces \cite[Rem~7.2.4.18]{HTT}, and
\item Alexandroff spaces, since the $\infty$-topos of sheaves associated to an Alexandroff space is equivalent to a presheaf $\infty$-topos.
\end{itemize}
\section{When is a sheaf compact?}\label{sec:when-compact}
Let $\mathcal C$ be a compactly generated stable $\infty$-category,
e.g. the unbounded derived category $\mathcal D(R)$ of a ring $R$ or
the $\infty$-category of spectra $\mathrm{Sp}$.
Given a sheaf $\mathscr F\in\Shv(X,\mathcal C)$, we define the \emph{support} of $\mathscr F$ by
$$
\supp\mathscr F = \lbrace 
x\in X\mid \mathscr F_x\not\simeq 0
\rbrace\subseteq X.
$$
As in \cite{neeman2001}, our study of the compact objects of $\Shv(X,\mathcal C)$ proceeds from an analysis of their supports.
By slightly adapting the proof of \cite[Lem~1.4]{neeman2001},
we get the following description of the support of a compact sheaf:
\begin{lem}\label{lem:cpt-supp}
  Let $X$ be a $\mathcal C$-hypercomplete
  locally compact Hausdorff space and let $\mathscr F\in\Shv(X,\mathcal C)^\omega$. Then the support $\supp\mathscr F$ is compact.
\end{lem}
\begin{proof}
  We first show that $\supp\mathscr F$ is contained in a compact subset of $X$.
  Consider the canonical map
  \begin{equation}\label{eq:precpt-filter}
    \colimit_U (j_U)_!j_U^*\mathscr F
    \to\mathscr F,
  \end{equation}
  where the colimit ranges over the poset of precompact open sets ordered by the rule $U\leq V$ if $\overline U\subseteq V$, and for each
  such $U$ we have denoted by $j_U\colon U\hookrightarrow X$ the inclusion.
  Since precompact open sets form a basis for the topology on $X$, the map~\eqref{eq:precpt-filter} is an equivalence of sheaves.
  Let $\phi\colon\mathscr F\xrightarrow{\sim} \colimit_U (j_U)_!(j_U)^*\mathscr F$ be some choice of inverse.
  Any finite union of precompact open sets is again precompact open, so the poset of precompact open sets is filtered. 
  Hence compactness of $\mathscr F$ implies that $\phi$ factors through $(j_U)_!j_U^*\mathscr F$ for some precompact open $U$, and it follows that $\supp\mathscr F$ is contained in a compact subset $\overline U\subseteq X$, as claimed.

  By the above, it remains only to be seen that $\supp\mathscr F$ is closed, or equivalently that its complement $X\setminus\supp\mathscr F$ is open.
  Suppose $x\in X\setminus{\supp\mathscr F}$. Then we have a recollement fiber sequence
  $$
  j_!j^*\mathscr F\to\mathscr F\to i_*i^*\mathscr F,
  $$
  where $j\colon X\setminus\lbrace x\rbrace\hookrightarrow X$
  and $i\colon\lbrace x\rbrace\hookrightarrow X$ are the inclusions, and since $x\not\in\supp\mathscr F$ we have $j_!j^*\mathscr F\simeq\mathscr F$.
  Since $j_!$ is a fully faithful left adjoint, it reflects compact objects, and we conclude that $j^*\mathscr F$ is again compact. But then $j^*\mathscr F$ is supported on a compact subset of $X\setminus\lbrace x\rbrace$ by the above, which must be closed as a subset of $X$, and hence $x$ lies in the interior of $X\setminus\supp\mathscr F$ as desired.
\end{proof}
\begin{lem}\label{lem:prop-pullback}
  If $f\colon X\to Y$ is a proper map of locally compact Hausdorff spaces,
  then the pullback functor $f^*$ preserves compact objects.
  In particular, if $X$ is a compact Hausdorff space and $E\in\mathcal C^\omega$, then $E_X\in\Shv(X,\mathcal C)^\omega$, where $E_X$ denotes the constant sheaf at $E$.
\end{lem}
\begin{proof}
  Since $f$ is proper, the pullback $f^*$ is left adjoint to $f_*\simeq f_!$, which is itself left adjoint to $f^!$. Hence $f_*$ preserves colimits, and it follows that its left adjoint $f^*$ preserves compact objects.
  The statement about constant sheaves follows by taking $f$ to be the projection from $X$ to a point.
\end{proof}
Our main result is the following description of the compact objects in $\Shv (X,\mathcal C)$:
\begin{thm}\label{thm:cpt-sheaves}
  Let $X$ be a $\mathcal C$-hypercomplete locally compact Hausdorff space. A sheaf $\mathscr F\in\Shv(X,\mathcal C)$ is compact if and only if
  \begin{enumerate}[label=(\roman*)]
  \item $\supp\mathscr F$ is compact;
  \item $\mathscr F$ is locally constant; and
  \item $\mathscr F_x\in\mathcal C^\omega$ for each $x\in X$.
  \end{enumerate}
\end{thm}
In particular, note that conditions (i) and (ii) together imply that if $\mathscr F$ is compact, then the support $\mathscr F$ must be a compact open subset of $X$.
On the other hand, (iii) guarantees that if $\mathscr F$ is constant on $U\subseteq X$, say with value $E$, then $E\in\mathcal C^\omega$. 
\begin{proof}
  `Necessity.' Suppose we are given $\mathscr F\in\Shv(X,\mathcal C)^\omega$.
  Then $\mathscr F$ must satisfy (i) by Lemma~\ref{lem:cpt-supp} and (ii) by Lemma~\ref{lem:prop-pullback}, since the stalk $\mathscr F_x$ at $x\in X$
  is the same as the pullback $i_x^*\mathscr F$ along the inclusion $i_x\colon\lbrace x\rbrace\hookrightarrow X$, which is always a proper map.
  It remains only to be seen that $\mathscr F$ is locally constant.
  Fix a point $x\in X$,
  and let $i_x$ again denote the inclusion of this point into $X$. Let $E = i_x^*\mathscr F$ denote the stalk of $\mathscr F$ at $x$.
  By \cite[Cor~7.1.5.6]{HTT}, there is an equivalence $\colimit_U\mathscr F(U)\simeq E$, where $U$ ranges over the poset of open neighborhoods of $x$.
  As $E$ is compact, this implies that
  $\mathscr F(U)\to E$ has a section for some $U$. Pick a precompact open neighborhood $W\ni x$ with $\overline W\subseteq U$, 
  and let $i\colon\overline W\hookrightarrow X$ denote the inclusion.
  As the canonical map $\mathscr F(U)\to E$ factors through the restriction $\mathscr F(U)\to(i^*\mathscr F)(\overline W)\to\mathscr F(W)$,
  the map $(i^*\mathscr F)(\overline W)\to E$ also admits a section $E\to (i^*\mathscr F)(\overline W)$.
  Viewing the latter as a morphism
  from the constant presheaf with value $E$ to $i^*\mathscr F$, 
  we get an induced map $\sigma\colon E_{\overline W}\to i^*\mathscr F$
  of sheaves over $\overline W$ which by construction induces an equivalence of stalks at $x$.
  Here both $E_{\overline W}$ and $i^*\mathscr F$ are compact, so the cofiber $\mathscr Q = \operatorname{cofib}(\sigma )$ is again compact.
  But then $\supp\mathscr Q$ is compact,
  so $W'=W\setminus\supp\mathscr Q$ is open
  and $\mathscr Q_x\simeq 0$ so $x\in W'$.
  \begin{figure}[h]
    \centering
    \def\svgwidth{0.5\textwidth}
    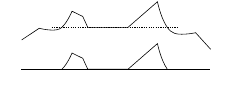
    \caption{`Espace \'etale' visualization of the fiber sequence $E_{\overline W}\to\mathscr F\to\mathscr Q$}
    \label{fig:espace-etale}
  \end{figure}  
  Furthermore, $\sigma$ restricts to an equivalence of sheaves on $W'$ by construction, so $\mathscr F\vert_{W'}$ is equivalent to the constant sheaf on $W'$ with value $E$, as desired.
  
  `Sufficiency.' Let $i\colon\supp\mathscr F\hookrightarrow X$ denote the inclusion. Since $i$ is both proper and an open immersion, both $i_*\simeq i_!$ and $i^*\simeq i^!$ preserve and reflect compact objects. We may therefore assume that $X$ is compact, after possibly replacing it with $\supp\mathscr F$.
  Pick a finite collection of closed subsets $Z_i\subseteq X$, $i=1,\dots ,n$, such that $\mathscr F$ is constant in a neighborhood of each $Z_i$ and such that $X$ is covered by the interiors $Z_i^{\mathrm o}$.
  Descent (Corollary~\ref{cor:descent-lem}) implies that the canonical functor
  \begin{align*}
    &\Shv(X,\mathcal C) \\
    &\qquad\to
      \displaystyle{
      \limit_{\Delta_{\leq n-1}}}
      \left(
      \begin{tikzcd}[ampersand replacement=\&,column sep=small]
        \Shv(\bigcap_1^nZ_i,\mathcal C)
        \arrow[r,shift left=10pt]
        \arrow[r,phantom,"\raisebox{1.5ex}{\vdots}"]
        \arrow[r,shift right=10pt]
        \&\cdots
        \arrow[r,shift left=10pt]
        \arrow[r,leftarrow,shorten <= 2pt,shorten >= 2pt,shift left=5pt]
        \arrow[r]
        \arrow[r,leftarrow,shorten <= 2pt,shorten >= 2pt,shift right=5pt]
        \arrow[r,shift right=10pt]
        \&\prod_{i,j}\Shv(Z_i\cap Z_j,\mathcal C)
        \arrow[r,shift left=5pt]
        \arrow[r,leftarrow,shorten <= 2pt,shorten >= 2pt]
        \arrow[r,shift right=5pt]
        \&\prod_i\Shv (Z_i,\mathcal C)
      \end{tikzcd}
      \right)
  \end{align*}
  is an equivalence.
  Write $I = \lbrace 1,\dots ,n\rbrace$ for short and put $Z_J = \bigcap_{j\in J}Z_j$ for each $J\subseteq I$.
  The canonical projection from $\Shv(X,\mathcal C)$ to 
  $\Shv(Z_J,\mathcal C)$
  is the restriction map.
  By construction, we have that 
  for each $J\subseteq I$, the restriction $\mathscr F\vert_{Z_J}$ is constant with value a compact object, and hence compact as an object of $\Shv(Z_J,\mathcal C)$ by the preceding lemma.
  According to \cite[Lem~6.3.3.6]{HTT}, the identity functor $\text{id}\colon\Shv (X,\mathcal C)\to\Shv (X,\mathcal C)$ is the colimit of a diagram $\Delta_{\leq n-1}\to\operatorname{Fun}(\Shv (X,\mathcal C),\Shv (X,\mathcal C))$ which sends the object $\lbrack k\rbrack\in\Delta_{\leq n-1}$ to the composition 
  $$
  \begin{tikzcd}
`    \Shv (X,\mathcal C)\arrow[r]\arrow[rr,bend left=20,"i_k^*",shorten <= 8pt ]
    &\displaystyle{\prod_{\substack{J\subseteq I, \\ |J|=k}}}\Shv (Z_J,\mathcal C)
    \arrow[r,phantom,"\simeq"]
    &\Shv \bigg(\displaystyle{\coprod_{\substack{J\subseteq I, \\ |J|=k}}} Z_J,\mathcal C\bigg)\arrow[r,"(i_k)_*"]
    & \Shv (X,\mathcal C),
  \end{tikzcd}
  $$
  and so for any filtered system $\lbrace\mathscr G_\alpha\rbrace_{\alpha\in A}$, we find
  \begin{align*}
    \operatorname{Map}(\mathscr F,\colimit_A\mathscr G_\alpha )
    &\simeq\limit_{\lbrack k\rbrack\in\Delta_{\leq n-1}}\operatorname{Map}(\mathscr F,(i_k)_*i_k^*\colimit\mathscr G_\alpha ) \\
    &\simeq
      \limit_{\lbrack k\rbrack\in\Delta_{\leq n-1}}\operatorname{Map}(i_k^*\mathscr F,\colimit_A i_k^*\mathscr G_\alpha ) \\
    &\simeq\limit_{\lbrack k\rbrack\in\Delta_{\leq n-1}}\colimit_A\operatorname{Map}(i_k^*\mathscr F, i_k^*\mathscr G_\alpha ) \\
    &\simeq\colimit_A\limit_{\lbrack k\rbrack\in\Delta_{\leq n-1}}\operatorname{Map}(i_k^*\mathscr F, i_k^*\mathscr G_\alpha ) \\
    &\simeq\colimit_A\operatorname{Map}(\mathscr F,\mathscr G_\alpha ),
  \end{align*}
  where the third equivalence uses that the restriction $i_k^*\mathscr F$ is compact\footnote{Indeed, we have already observed that $\mathscr F\vert_{Z_J}$ is compact for each $J$, and hence the associated object $i_k^*\mathscr F$ in the product $\Pi_J\Shv (Z_J,\mathcal C)$is also compact according to \cite[Lem~5.3.4.10]{HTT}.} and the second-last equivalence uses that filtered colimits commute are left exact in $\mathcal S$.
\end{proof}
As a corollary, we recover Neeman's result:
\begin{cor}[Neeman]
  Let $M$ be a non-compact connected manifold. Then $\mathscr F\in\Shv (M,\mathcal C)^\omega$ if and only if $\mathscr F\simeq 0$.
\end{cor}
In fact our result shows that the conclusion of Neeman's result holds more generally if $M$ is replaced by a $\mathcal C$-hypercomplete locally compact Hausdorff space $X$ whose quasicomponents are all non-compact.

\medskip

As a further corollary to our theorem, we can also describe the dualizable objects of, say, the $\infty$-category of $\text{Mod}_R$-valued sheaves on a $\text{Mod}_R$-hypercomplete locally compact Hausdorff space $X$, where $R$ is some $\mathbb E_\infty$-ring. As with the compact objects, the dualizable objects turn out to be very sparse:
\begin{cor}
  Let $\mathrm{Mod}_R^\otimes$ be the symmetric monoidal $\infty$-category of modules over an $\mathbb E_\infty$-ring $R$, and let $X$ be a $\mathrm{Mod}_R$-hypercomplete locally compact Hausdorff space. 
  With respect to the induced symmetric monoidal structure on $\Shv (X,\mathrm{Mod}_R)$, 
  a sheaf $\mathscr F\in\Shv (X,\mathrm{Mod}_R)$ is dualizable
  if and only if
  \begin{enumerate}[label=(\roman*)]
  \item $\mathscr F$ is locally constant, and
  \item $\mathscr F_x$ is a perfect $R$-module\footnote{i.e. a compact object of $\mathrm{Mod}_R$} for each $x\in X$.
  \end{enumerate}
\end{cor}
\begin{proof}
  `Sufficiency.' Since $\Shv (X,\mathrm{Mod}_R)^\otimes$ is closed symmetric monoidal, it suffices to show that for each sheaf $\mathscr F$ satisfying the two conditions, the canonical map
  \begin{equation}\label{eq:dual-test}
    \mathscr{H}om (\mathscr F,R_X)\otimes\mathscr F\to\mathscr Hom (\mathscr F,\mathscr F)
  \end{equation}
  is an equivalence,
  where $R_X$ is the constant sheaf at $R$ and $\mathscr Hom ({-},{-})$ denotes the internal mapping object in $\Shv (X,\mathrm{Mod}_R)$.
  For sufficiently small open subsets $U\subseteq X$, the restriction $\mathscr F\vert_U$ is equivalent to the constant sheaf $F_U = \pi^*F$ at a perfect $R$-module $F$, 
  and since pullback--being symmetric monoidal--preserves dualizable sheaves, we find that~\eqref{eq:dual-test} restricts to an equivalence on $U$. Since~\eqref{eq:dual-test} is a morphism of sheaves which is locally an equivalence, it must be an equivalence, proving the claim.

  `Necessity.' Assume that $\mathscr F$ is dualizable, and let $x\in X$ be some point. The condition on the stalks is immediate, since pullback preserves dualizable sheaves.
  We must show that $\mathscr F$ is locally constant in a neighborhood of $x$. Pick a precompact open neighborhood $U\ni x$. Then $\mathscr F\vert_{\overline U}$ is again dualizable, and since the monoidal unit $R_{\overline U} = \pi^*R\in\Shv (\overline U,\mathrm{Mod}_R)$ is compact, it follows that $\mathscr F\vert_{\overline U}$ is compact as an object of $\Shv (\overline U,\mathrm{Mod}_R)$. But then the previous theorem implies that it must be locally constant on $\overline U$, and hence also on the subset $U$ as desired.
\end{proof}
\section{When is $\Shv(X,\mathcal C)$ compactly generated?}
In this section, we prove the following characterization of those locally compact Hausdorff spaces $X$ that have $\Shv (X,\mathcal C)$ compactly generated:
\begin{prop}\label{prop:when-cpt}
  Let $\mathcal C$ be a non-trivial compactly generated stable $\infty$-category, and
  let $X$ be a $\mathcal C$-hypercomplete locally compact Hausdorff space. Then $\Shv (X,\mathcal C)$ is compactly generated if and only if $X$ is totally disconnected.
\end{prop}

\subsection{Proof of the proposition}
The proof will use the following observation:
\begin{lem}\label{lem:small-compare}
  Let $\mathcal C$ be a compactly generated stable $\infty$-category,
  and let $\lbrace C_i\rbrace_{i\in I}$ and $\lbrace D_i\rbrace_{i\in I}$ be filtered systems in $\mathcal C$ indexed over the same poset $I$.
  \begin{enumerate}[label=(\arabic*)]
  \item Suppose that for each $i\in I$, there is some $j\geq i$ so that the transition map $C_i\to C_j$ factors through the zero object $*$.
    Then $\colimit_I C_i\simeq *$. If each $C_i$ is compact, then the converse holds.
  \item Suppose that for each comparable pair $i\leq j$ in $I$ there are horizontal equivalences making
  $$
  \begin{tikzcd}
    C_i\arrow[r,dashed,"\simeq"]
    \arrow[d] 
    & D_i\arrow[d] \\
    C_j\arrow[r,dashed,"\simeq"]
    & D_j
  \end{tikzcd}
  $$
  commute, where the vertical maps are the transition maps.
  If each $C_i$ is compact, then $\colimit_IC_i\simeq *$ if and only if $\colimit_ID_i\simeq *$.
  \end{enumerate}
\end{lem}
\begin{proof}
  Note that (2) follows from (1), since the existence of such commutative squares implies that $\lbrace C_i\rbrace_I$ has the vanishing property for transition maps described in (1) if and only if $\lbrace D_i\rbrace_I$ has that property.
  
  For the first claim in (1), it suffices to show that $\operatorname{Map}_{\mathcal C}(D,\colimit_{i\in I}C_i)$ is contractible
  for each compact object $D\in\mathcal C^\omega$.
  For this, first observe that
  $$
  \pi_0\operatorname{Map}_{\mathcal C}(D,\colimit_{i\in I}C_i)\cong\colimit_{i\in I}\pi_0\operatorname{Map}(D,C_i)\cong *,
  $$
  since our assumption guarantees that any homotopy class $D\to C_i$ is identified $D\to *\to C_i$ after postcomposing with the transition map $C_i\to C_j$
  for sufficiently large $j\geq i$. Applying the same argument for the compact object $\Sigma^nD$, $n\geq 1$, we find that
  $$
  \pi_n\operatorname{Map}_{\mathcal C}(D,\colimit_{i\in I}C_i)\cong\pi_0\operatorname{Map}_{\mathcal C}(\Sigma^nD,\colimit_{i\in I}C_i)
  $$
  vanishes also.
  
  Assume now that each $C_i$ is compact and that $\colimit_I C_i\simeq *$. Then
  $$
  \colimit_{j\in I}\operatorname{Map}_{\mathcal C}(C_i,C_j)
  \simeq
  \operatorname{Map}_{\mathcal C}(C_i,\colimit_{j\in I} C_j),
  $$
  and since $\pi_0$ commutes with filtered colimits of anima, it follows that for sufficiently large $j\geq i$ the transition map $C_i\to C_j$ is homotopic to $C_i\to *\to C_j$.
\end{proof}
\begin{proof}[Proof of Proposition~\ref{prop:when-cpt}]
  `Sufficiency.' The $\infty$-category of sheaves of anima 
  $\Shv (X)$ is compactly generated by \cite[Prop~6.5.4.4]{HTT},
  and hence so is $\Shv(X,\mathcal C)\simeq\Shv (X)\otimes\mathcal C$ according to \cite[Lem~5.3.2.11]{HA}.
  
  `Necessity.' Let $x\in X$. We must show that if $y\in X$ lies in the same connected component as $X$, then $y=x$. For this, pick an object $C\not\simeq 0$ in $\mathcal C$ and
  let $x_*C$ denote the skyscraper sheaf at $x$ with value $C$. By assumption there is a filtered system $\lbrace\mathscr F_i\rbrace_{i\in I}$ of compact sheaves with $\colimit_I\mathscr F_i\simeq x_*C$.
  For each $i$, the fact that $\mathscr F_i$ is locally constant and that $x$ and $y$ lie in the same connected component means there is a non-canonical equivalence of stalks $x^*\mathscr F_i\simeq y^*\mathscr F_i$. One should not expect to find a system of such non-canonical equivalences assembling into a natural transformation, essentially because the neighborhoods on which the $\mathscr F_i$ are constant could get smaller and smaller as $i$ increases. Nevertheless, given a comparable pair $i\leq j$ in $I$, one can pick equivalences making the diagram
  \begin{equation}\label{eq:small-transf}
    \begin{tikzcd}
      x^*\mathscr F_i\arrow[r,dashed,"\simeq"] 
      \arrow[d]
      & y^*\mathscr F_i \arrow[d]\\
      x^*\mathscr F_j\arrow[r,dashed,"\simeq"] 
      & y^*\mathscr F_j
    \end{tikzcd}
  \end{equation}
  where the vertical maps are the transition maps.
  To see this, simply note that the set of $z\in Z$ for which there is a commutative diagram
  $$
  \begin{tikzcd}
    x^*\mathscr F_i\arrow[r,dashed,"\simeq"] 
    \arrow[d]
    & z^*\mathscr F_i \arrow[d]\\
    x^*\mathscr F_j\arrow[r,dashed,"\simeq"] 
    & z^*\mathscr F_j
  \end{tikzcd}
  $$
  is a clopen subset of $X$, since any point admits a neighborhood on which both $\mathscr F_i$ and $\mathscr F_j$ are constant.
  Since all of the $\mathscr F_i$ have compact stalks by Theorem~\ref{thm:cpt-sheaves}, it follows from Lemma~\ref{lem:small-compare} that the stalk $(x_*C)_y\simeq\colimit_Iy^*\mathscr F_i$ is nonzero. But $X$ is Hausdorff, so this implies that $y=x$ as desired.
\end{proof}
\begin{rmk}
  Lemma~\ref{lem:small-compare} is also true if $\mathcal C$ is any ordinary category,
  e.g. the category of abelian groups $\mathrm{Ab}$.
  It is illuminating to consider why the lemma holds in this concrete setting.
  Given a filtered system of abelian groups $\lbrace A_i\rbrace_{i\in I}$,
  the associated colimit can be described as the quotient of $\bigoplus_I A_i$
  by the subgroup consisting of elements $a\in A_i$ such that there is $j\geq i$
  for which the transition map $\varphi_{ij}\colon A_i\to A_j$ maps $a$ to zero.
  Clearly $\colimit_IA_i\cong 0$ is implied by the assumption that for every $i\in I$
  there is $j\geq i$ with $\varphi_{ij}\colon A_i\to A_j$ being zero.
  For the partial converse, assume now that each $A_i$ is a compact object of
  $\mathrm{Ab}$, i.e. a finitely generated abelian group,
  and that $\colimit_IA_i\cong 0$. Let $i\in I$ and pick a generating set $a_1,\dots ,a_n$ for $A_i$.
  Since $\colimit_IA_j\cong 0$, there is $j_1,\dots ,j_n\geq i$ with $\varphi_{ij_s}(a_s) = 0$ for each $s$.
  Using that $I$ is filtered, pick $j\in I$ so that $j\geq j_s$ for each $s$.
  Then $\varphi_{ij}(a_s) = \varphi_{j_sj}\varphi_{ij_s}(a_s) = 0$ for each $s$, and hence $\varphi_{ij} = 0$.
\end{rmk}
\subsection{Hausdorff schemes}
Unlike in point-set topology, compactly generated categories of sheaves are abundant in algebraic geometry.
Using results of Hochster~\cite{hochster1969prime}, Proposition~\ref{prop:when-cpt} can be interpreted as saying that $\Shv (X,\mathcal C)$ is only compactly generated when $X$ happens to come from algebraic geometry:
\begin{prop}
  Let $\mathcal C$ be a nontrivial compactly generated stable $\infty$-category,
  and let $X$ be a $\mathcal C$-hypercomplete locally compact Hausdorff space
  Then $\Shv (X,\mathcal C)$ is compactly generated if and only if $X$ is the underlying space of a scheme
\end{prop}
Indeed, a locally compact Hausdorff space is totally disconnected if and only if it admits a basis of compact open sets if and only if it is the underlying space of a scheme. The second equivalence is the Hausdorff case of \cite[Thm~9]{hochster1969prime}. 
  
  For the first equivalence, note in one direction that if $X$ admits a basis of compact open sets, then every $x\in X$ has $\lbrace x\rbrace = \bigcap_{U\ni x} U$, with $U$ ranging over compact open neighborhoods of $x$. Since each compact open neighborhood is clopen, we thus have that $\lbrace x\rbrace$ is a quasi-component in $X$, and hence that $X$ is totally disconnected.
  
  For the other direction, we must show that for every $x\in X$ and every open neighborhood $V\ni x$, there is a compact open $W$ with $x\in W\subseteq V$. Since $X$ is locally compact, we may assume that $V$ is precompact. By assumption $\lbrace x\rbrace = \bigcap_{U\ni x}U$, with $U$ ranging over clopen neighborhoods of $x$. Since each of these $U$ is in particular closed, we have that each $U\cap\partial\overline V$ is compact. By the finite intersection property, it therefore follows from $\bigcap_{U\ni x} U\cap \partial\overline{V} = \varnothing$ that for small enough clopen $U\ni x$, $U\cap\partial\overline{V} = \varnothing$. Hence $U\cap\overline V = U\cap V$ is a compact open neighborhood of $x$ contained in $V$, as desired.

  \subsection{When is $\Shv (X)$ compactly generated?}
Proposition~\ref{prop:when-cpt} says that the $\infty$-category of sheaves on $X$ with coefficients in a stable $\infty$-category is rarely compactly generated when $X$ is a locally compact Hausdorff space.
If we had asked the same question `without coefficients,'
this would have been an easier observation:
\begin{prop}
  Let $X$ be a quasi-separated\footnote{Recall that a topological space $X$ is said to be \emph{quasi-separated} if for any pair of compact open subsets $U,V\subseteq X$, the intersection $U\cap V$ is again compact. Note that all Hausdorff spaces are quasi-separated.} topological space.
  The $\infty$-topos $\Shv (X)$ of sheaves of anima on $X$ is compactly generated if and only if the sobrification of $X$ is the underlying space of a scheme.
\end{prop}
\begin{proof}
  One direction is \cite[Thm~7.2.3.6]{HTT}. For the other direction, assume that $\Shv (X)$ is compactly generated. 
  Then so is the frame 
  $\mathcal U\simeq\tau_{\leq -1}\Shv (X)$ of open subsets of $X$ by \cite[Cor~5.5.7.4]{HTT}. But this means that $X$ admits a basis of compact open sets,
  and hence the sobrification of $X$ is the underlying space of a scheme according to \cite[Thm~9]{hochster1969prime}.
\end{proof}
\appendix
\section{Descent for maps with local sections} 
In this short appendix, we prove a descent lemma that was used in the proof of Theorem~\ref{thm:cpt-sheaves},
which is an immediate generalization of \cite[Cor~4.1.6]{techniques}.

\medskip

Let $\mathcal C$ be a compactly generated $\infty$-category and let $f\colon X\to Y$ be a continuous map of topological spaces.    
Recall that the \emph{\v Cech nerve} of $f$ is the augmented simplicial topological space $X_\bullet$ with $X_{-1} = Y$ and $p$-simplices
$$
X_p = \underbrace{X\times_Y \dots\times_Y X}_{p\text{ times}}
$$
for $p\geq 0$, with face maps given by projections and degeneracy maps given in the obvious way. More formally, if $\Delta_+$ is the category of finite (possibly empty) ordinals and $\mathcal T\mathrm{op}$ is the category of topological spaces, then $X_\bullet\colon \Delta^{\text{op}}_+\to\mathcal T\mathrm{op}$ is
defined by right Kan extending $(f\colon X\to Y)\colon \Delta_{+,\leq 0}^{\text{op}}\to\mathcal T\mathrm{op}$ along the inclusion functor $\Delta_{+,\leq 0}^{\text{op}}\subset \Delta^{\text{op}}_+$.

Letting $\Shv^*({-},\mathcal C)$ denote the contravariant functor from $\mathcal T\mathrm{op}$ to $\inftyCats$ given informally by $X\mapsto \Shv (X,\mathcal C)$ on objects and $f\mapsto f^*$ on morphisms, we then have the following useful definition:
\begin{defn}
  The function $f$ is of \emph{$\mathcal C$-descent type} if the canonical functor
  $$
  \Shv (X,\mathcal C)\to\limit_{\Delta}\Shv^*(X_\bullet ,\mathcal C)
  $$
  is an equivalence.
\end{defn}
Let us say that $f$ \emph{admits local sections} if for every $x\in X$, there is an open set $U\ni x$ such that the restriction $f\colon f^{-1}(U)\to U$ admits a section.
\begin{prop}
  If $f$ admits local sections, then $f$ is of $\mathcal C$-descent type.
\end{prop}
\begin{proof}
  By ordinary \v Cech descent, we may assume that $f$ admits a section globally on $X$, after possibly passing to an open cover of $X$ on which this is true. Let $\varepsilon\colon Y\to X$ be a choice of such a section. The section $\varepsilon$ allows us to endow the \v Cech nerve $X_\bullet$ with the structure of a split augmented simplicial space, by defining the extra degeneracies $h_i\colon X_p\to X_{p+1}$ by
  $$
  h_i(x_0,\dots ,x_p) = (x_0,\dots ,x_{i-1},\varepsilon (y),x_i,\dots ,x_p)
  $$
  where $y = f(x_0) = \dots = f(x_p)$. It then follows that the split coaugmented cosimplicial $\infty$-category $\Shv^*(X_\bullet ,\mathcal C)$ is a limit diagram by \cite[Lem~6.1.3.16]{HTT}
\end{proof}
\begin{cor}\label{cor:descent-lem}
  Let $\lbrace A_i\rbrace_{i\in I}$ be a collection of subsets of $X$ such that $X = \bigcup_I A_i^{\mathrm o}$, where $A_i^{\mathrm o}$ is the interior of $A_i$. Then the canonical map $\coprod_I A_i\to X$ is of $\mathcal C$-descent type.
\end{cor}
\begin{proof}
  The canonical map $\coprod_IA_i\to X$ admits a section on $A_j^{\mathrm o}$ given by $A_j^{\mathrm o}\hookrightarrow A_j\to\coprod_IA_i$, where the second map is the canonical injection.
\end{proof}
\bibliographystyle{alpha}
\bibliography{sources}
\end{document}

%% file: espace-etale.pdf_tex
\begingroup%
  \makeatletter%
  \providecommand\color[2][]{%
    \errmessage{(Inkscape) Color is used for the text in Inkscape, but the package 'color.sty' is not loaded}%
    \renewcommand\color[2][]{}%
  }%
  \providecommand\transparent[1]{%
    \errmessage{(Inkscape) Transparency is used (non-zero) for the text in Inkscape, but the package 'transparent.sty' is not loaded}%
    \renewcommand\transparent[1]{}%
  }%
  \providecommand\rotatebox[2]{#2}%
  \newcommand*\fsize{\dimexpr\f@size pt\relax}%
  \newcommand*\lineheight[1]{\fontsize{\fsize}{#1\fsize}\selectfont}%
  \ifx\svgwidth\undefined%
    \setlength{\unitlength}{113.38582677bp}%
    \ifx\svgscale\undefined%
      \relax%
    \else%
      \setlength{\unitlength}{\unitlength * \real{\svgscale}}%
    \fi%
  \else%
    \setlength{\unitlength}{\svgwidth}%
  \fi%
  \global\let\svgwidth\undefined%
  \global\let\svgscale\undefined%
  \makeatother%
  \begin{picture}(1,0.375)%
    \lineheight{1}%
    \setlength\tabcolsep{0pt}%
    \put(0,0){\includegraphics[width=\unitlength,page=1]{espace-etale.pdf}}%
    \put(0.76241147,0.26853574){\color[rgb]{0,0,0}\makebox(0,0)[lt]{\lineheight{1.25}\smash{\begin{tabular}[t]{l}$E_{\overline{W}}$\end{tabular}}}}%
    \put(0.4706104,0.09854329){\color[rgb]{0,0,0}\makebox(0,0)[lt]{\lineheight{1.25}\smash{\begin{tabular}[t]{l}$x$\end{tabular}}}}%
    \put(0.40967204,0.19073194){\color[rgb]{0,0,0}\makebox(0,0)[lt]{\lineheight{1.25}\smash{\begin{tabular}[t]{l}$E$\end{tabular}}}}%
    \put(0.28104655,-0.02935664){\color[rgb]{0,0,0}\makebox(0,0)[lt]{\lineheight{1.25}\smash{\begin{tabular}[t]{l}$W'=W\setminus\supp\mathscr Q$\end{tabular}}}}%
    \put(0.9030627,0.15421193){\color[rgb]{0,0,0}\makebox(0,0)[lt]{\lineheight{1.25}\smash{\begin{tabular}[t]{l}$\mathscr F$\end{tabular}}}}%
    \put(0.70610587,0.11750081){\color[rgb]{0,0,0}\makebox(0,0)[lt]{\lineheight{1.25}\smash{\begin{tabular}[t]{l}$\mathscr Q$\end{tabular}}}}%
    \put(0.89880979,0.06937726){\color[rgb]{0,0,0}\makebox(0,0)[lt]{\lineheight{1.25}\smash{\begin{tabular}[t]{l}$X$\end{tabular}}}}%
    \put(0,0){\includegraphics[width=\unitlength,page=2]{espace-etale.pdf}}%
  \end{picture}%
\endgroup%

%% file: cpt-sheaves.bbl
\begin{thebibliography}{Hoy18}

\bibitem[CS23]{cesnavicius-scholze}
Kestutis Cesnavicius and Peter Scholze.
\newblock Purity for flat cohomology, 2023.
\newblock Preprint arXiv:\href{https://arxiv.org/abs/1912.10932}{1912.10932}.

\bibitem[Efi22]{efimov-ICM}
Alexander Efimov.
\newblock On the {$K$}-theory of large triangulated categories.
\newblock Available at \url{https://www.youtube.com/watch?v=RUDeLo9JTro}, 2022.
\newblock Talk at ICM.

\bibitem[Hoc69]{hochster1969prime}
Melvin Hochster.
\newblock Prime ideal structure in commutative rings.
\newblock {\em Trans. Amer. Math. Soc.}, 142:43--60, 1969.

\bibitem[Hoy16]{hoyois-cw}
Marc Hoyois.
\newblock Is the {$\infty$-topos $Sh(X)$ hypercomplete whenever $X$ is a CW
  complex?}
\newblock MathOverflow \url{https://mathoverflow.net/q/247061}, 2016.
\newblock Accessed 24-08-2023.

\bibitem[Hoy18]{hoyois}
Marc Hoyois.
\newblock {$K$}-theory of dualizable categories (after {A.~Efimov}).
\newblock Available at
  \url{https://hoyois.app.uni-regensburg.de/papers/efimov.pdf}, 2018.
\newblock Accessed 24-07-2023.

\bibitem[Lur09]{HTT}
Jacob Lurie.
\newblock {\em Higher {T}opos {T}heory}.
\newblock Annals of Mathematics Studies. Princeton University Press, 2009.

\bibitem[Lur17]{HA}
Jacob Lurie.
\newblock Higher {A}lgebra.
\newblock Available at \url{https://www.math.ias.edu/~lurie/papers/HA.pdf},
  2017.
\newblock Accessed 24-07-2023.

\bibitem[Lur18]{SAG}
Jacob Lurie.
\newblock Spectral {A}lgebraic {G}eometry.
\newblock Available at
  \url{https://www.math.ias.edu/~lurie/papers/SAG-rootfile.pdf}, 2018.
\newblock Accessed 24-07-2023.

\bibitem[Nee01]{neeman2001}
Amnon Neeman.
\newblock On the derived category of sheaves on a manifold.
\newblock {\em Doc. Math.}, 6:483--488, 2001.

\bibitem[Sch]{scholze-6FF}
Peter Scholze.
\newblock {Six-Functor Formalisms}.
\newblock Available at
  \url{https://people.mpim-bonn.mpg.de/scholze/SixFunctors.pdf}.
\newblock Accessed 24-08-2023.

\bibitem[SD72]{techniques}
Bernard Saint-Donat.
\newblock Techniques de descente cohomologique.
\newblock In {\em Th{\'e}orie des Topos et Cohomologie Etale des Sch{\'e}mas:
  Tome 2}, pages 83--162. Springer, 1972.

\bibitem[Ver65]{verdier1965dualite}
Jean-Louis Verdier.
\newblock Dualit{\'e} dans la cohomologie des espaces localement compacts.
\newblock {\em S{\'e}minaire Bourbaki}, 9:337--349, 1965.

\bibitem[Vol23]{volpe23}
Marco Volpe.
\newblock The six operations in topology, 2023.
\newblock Preprint arXiv:\href{https://arxiv.org/abs/2110.10212}{2110.10212}.

\end{thebibliography}
